\documentclass[twoside, leqno]{article}

\usepackage{amssymb}
\usepackage{amsmath}
\usepackage{amsthm}
\usepackage{color}
\usepackage{mathrsfs}

\allowdisplaybreaks

\def\rr{{\mathbb R}}

\def\nn{{\mathbb N}}

\def\cb{{\mathcal B}}

\def\fz{\infty}

\def\ez{\epsilon}
\def\gz{{\gamma}}

\def\lf{\left}
\def\r{\right}

\def\wz{\widetilde}

\def\loc{{\mathop\mathrm{\,loc\,}}}

\def\bint{{\ifinner\rlap{\bf\kern.35em--}
\int\else\rlap{\bf\kern.45em--}\int\fi}\ignorespaces}

\def\bbint{{\ifinner\rlap{\bf\kern.35em--}
\hspace{0.078cm}\int\else\rlap{\bf\kern.45em--}\int\fi}\ignorespaces}

\newtheorem{thm}{Theorem}[section]
\newtheorem{lem}[thm]{Lemma}
\newtheorem{cor}[thm]{Corollary}
\newtheorem{prop}[thm]{Proposition}
\theoremstyle{definition}
\newtheorem{rem}[thm]{Remark}
\newtheorem{defn}[thm]{Definition}

\numberwithin{equation}{section}

\numberwithin{equation}{section}

\begin{document}

\arraycolsep=1pt

\title{\bf\Large Haj\l asz Gradients Are Upper Gradients
\footnotetext{\hspace{-0.35cm} 2010 \emph{Mathematics Subject Classification}. Primary 30L99;
Secondary 46G05, 46E35, 46E30, 42B35, 42B37.
\endgraf
{\it Key words and phrases}. Haj\l asz gradient, upper gradient,
metric measure space, Sobolev space
\endgraf
Renjin Jiang, Dachun Yang and Wen Yuan are partially supported
by the National Natural Science Foundation of China
(Nos. 11301029, 11171027, 11361020 and 11101038),
the Specialized Research Fund for the Doctoral Program of Higher Education
of China (No. 20120003110003) and the Fundamental Research Funds for Central
Universities of China (Nos. 2012LYB26, 2013YB60 and 2012CXQT09).
Nageswari Shanmugalingam is partially supported by
grant \#DMS-1200915 from the National Science Foundation (U.S.A.).
}}
\author{Renjin Jiang, Nageswari Shanmugalingam, Dachun Yang\,\footnote{Corresponding
author.}\ \ and Wen Yuan}
\date{}
\maketitle

\vspace{-0.8cm}

\begin{center}
\begin{minipage}{13cm}
{\small {\bf Abstract}\quad
Let $(X, d, \mu)$ be a metric measure space, with $\mu$ a Borel regular measure.
In this article, we prove that, if $u\in L^1_\loc(X)$ and $g$ is a
Haj\l asz gradient of $u$, then there exists $\widetilde u$ such that
$\widetilde u=u$
almost everywhere and $4g$ is a $p$-weak
upper gradient of $\widetilde u$. This result avoids a priori assumption
on the quasi-continuity of $u$ used in [Rev. Mat. Iberoamericana 16 (2000), 243-279].
We also introduce
the notion of local Haj\l asz gradients, and investigate the relations between
the local Haj\l asz gradient
and the upper gradient.}
\end{minipage}
\end{center}

\vspace{0.1cm}

\section{Introduction\label{s1}}
\hskip\parindent
As a substitute for the classical gradient
in metric measure spaces, the Haj\l asz gradients were first introduced by
Haj\l asz \cite{h96} in 1996. This opened the door to
the study of Sobolev spaces on metric measure spaces.
Let $(X,d,\mu)$ be a metric measure space
with a {nontrivial Borel regular measure} $\mu$, which
is finite on bounded sets and positive on open sets.

\begin{defn}
Given a measurable function $u$  on $X$,
a non-negative measurable function $g$ on $X$ is called a
\textit{Haj\l asz gradient} of $u$
if there is a set $E\subset X$ with $\mu(E)=0$
such that, for all $x,\ y\in X\setminus E$,
\begin{equation}\label{HG}
|u(x)-u(y)|\leq d(x,y)[g(x)+g(y)].
\end{equation}
\end{defn}

In~\cite{h96} the above notion was employed  to introduce the Sobolev space
$M^{1,p}(X)$ for $p\in(1,\fz)$ on a metric measure space $(X,d,\mu)$. The Haj\l asz-Sobolev space $M^{1,p}(X)$
is defined to be the set of all functions $u\in L^p(X)$ having a  Haj\l asz gradient
$g\in L^p(X)$. The norm of this space is given by
$$\|u\|_{M^{1,p}(X)}:=\|u\|_{L^p(X)}+\inf_g\|g\|_{L^p(X)},$$
where the infimum is taken over all Haj\l asz gradients of $u$.

In 1998,  Heinonen and Koskela \cite{hk98} introduced
another type of gradients in metric measure spaces called upper gradients. Recall that
 a \emph{rectifiable curve} $\gamma$ is a continuous
mapping from an interval into $X$ with finite length.
In what follows,
$|\gz|$ denotes the image of $\gz$ in $X$.
The \textit{p-modulus} of a collection $\Gamma$ of curves
is defined by
\begin{equation*}
{\rm Mod}_p(\Gamma):=\inf_{\rho\in F(\Gamma)}\|\rho\|_{L^p(X)}^p,
\end{equation*}
where $F(\Gamma)$ is the collection of all non-negative Borel measurable functions $\rho$
on $X$ that satisfy $\int_\gamma\rho(s)\, ds\ge 1$ for each rectifiable curve $\gamma\in \Gamma$.

\begin{defn}
Let $u$ be a measurable function on $X$.
A non-negative Borel function $g$ on $X$ is called an \emph{upper gradient} of $u$ if
\begin{equation}\label{UG}
|u(\gz(a))-u(\gz(b))|\le \int_\gz g(s)\,ds
\end{equation}
holds true for all non-constant rectifiable curves $\gz:\ [a,b]\to X$.
If the inequality \eqref{UG} holds true for all non-constant rectifiable
curves in $X$ except a family of
curves of $p$-modulus zero, where $p\in [1,\fz)$, then
$g$ is called a \textit{$p$-weak upper gradient} of $u$.
\end{defn}

In the definition of  $p$-weak upper gradients, we may relax the Borel measurability of $g$ to
$\mu$-measurability of $g$; see, for instance, \cite[Section 1.5]{bb}.
 This is because by the assumption that $\mu$ is Borel regular, we have that, whenever $g$ is
$\mu$-measurable, there is a Borel measurable function $\wz g\ge g$ such that $\wz g=g$ almost everywhere.

Using the notion of
$p$-weak upper gradients introduced in
\cite{km98},
the Newton-Sobolev space (also called the Newtonian space)
$N^{1,p}(X)$, $p\in[1,\fz)$, on a metric measure space $X$
was introduced in~\cite{s00}.
Recall that the space $\wz N^{1,p}(X)$ is defined
to be the set of all functions $u\in L^p(X)$ having a $p$-weak upper gradient
$g\in L^p(X)$, equipped with the quasi-norm
$$\|u\|_{\wz N^{1,p}(X)}:=\|u\|_{L^p(X)}+\inf_g\|g\|_{L^p(X)},$$
where the infimum is taken over all $p$-weak upper gradients of $u$.
The \emph{Newton-Sobolev space} $N^{1,p}(X)$ is the quotient
space $\wz N^{1,p}(X)/\sim$, where the relation $\sim$ is defined by $u\sim v$ if and only if
$\|u-v\|_{\wz N^{1,p}(X)}=0$.

It is natural to compare these two notions of gradients. As shown by~\cite[Lemma 4.7]{s00},
a Haj\l asz gradient of a continuous function $u$, up to some modifications
on a set of measure zero, is an upper gradient; see also \cite{h03,hkst}. On the other hand, it is known that
an upper gradient may not be a  Haj{\l}asz gradient, even if the underlying space is well connected.
In general, one should think of a  Haj\l asz gradient of a function as the Hardy-Littlewood maximal
function of an upper gradient, if the underlying space supports a Poincar\'e inequality.

Based on this fact and the density of continuous functions
in the Haj{\l}asz-Sobolev spaces, it was shown in~\cite{s00} that Haj{\l}asz-Sobolev spaces are continuously
embedded into Newton-Sobolev spaces (see also \cite[Theorem 8.6]{h03}).
Furthermore, the argument there can easily be generalized to show the embedding from
function spaces defined via Haj{\l}asz gradients to the corresponding spaces based on upper gradients,
provided that {\it continuous functions are dense in the  Haj{\l}asz type spaces}; see,
for example, \cite{CJ,h03,hhp,Tu}. The above discussion holds true for Haj\l asz and Newtonian spaces based on
the function spaces $L^p(X)$. The goal of this article is to address the issue of whether Haj\l asz and
Newtonian spaces based on other types of quasi-Banach spaces, e.\,g. Morrey spaces, coincide,
in which whether continuous functions are dense or not is not known.

Our main result below shows that, in general, Haj{\l}asz gradients are upper gradients.

\begin{thm}\label{t1.1}
Let $u$, $g\in L^1_\loc(X)$. Suppose that $g$ is a Haj\l asz gradient of $u$.
Then there exist $\wz u,\, \wz g\in L^1_\loc(X)$ such that $u=\wz u$ and $g=\wz g$ almost everywhere and
$4\,\wz g$ is an upper gradient of $\wz u$.
\end{thm}

The proof of Theorem \ref{t1.1} is given in Section 2 and, in Section 3, we first introduce
the notion of local Haj\l asz gradients (see Definition~\ref{lhg}) and further show that
local Haj\l asz gradients are upper gradients. The key tool used to
prove Theorem~\ref{t1.1} is McShane's extension of
Lipschitz functions defined on subsets of $X$ (see Lemma~\ref{l2} below).
In Section 4, we apply Theorem~\ref{t1.1} to several concrete settings, including Morrey-Sobolev spaces
in which Lipschitz functions are not dense; see~\cite{lyy}.

\section{Proof of Theorem \ref{t1.1}}

\hskip\parindent
To prove Theorem \ref{t1.1}, we need two technical lemmas.
The following one follows from an easy argument; see \cite{s00} and \cite[Lemma 9.2.2]{hkst}.

\begin{lem}[Refinement of Haj\l asz gradients]\label{l1}
Let $u$, $g\in L^1_\loc(X)$. Suppose that $g$ is a Haj\l asz gradient of $u$.
Then there exists  a Borel function $\widehat g\in L^1_\loc(X)$ such that
$g=\widehat g$ almost everywhere and, for all $x,\,y\in X$,
\[
  |u(x)-u(y)|\le d(x, y)[\widehat g(x) + \widehat g(y)].
\]
\end{lem}

\begin{proof}
From the definition of Haj\l asz gradient, there is a set $E\subset X$ with $\mu(E)=0$
such that, for all $x,\ y\in X\setminus E$,
\begin{equation*}
|u(x)-u(y)|\leq d(x,y)[g(x)+g(y)].
\end{equation*}
Choosing $\widehat g=g_0\chi_{X\setminus E_0}+\infty \chi_{E_0}$, where $E_0\supset E$ is a Borel set such that
$\mu(E_0)=0$ and $g_0\ge g$ is a Borel function such that $g=g_0$ almost everywhere
(see~\cite[Proposition 1.1]{bb}), gives the desired Haj\l asz gradient.
\end{proof}

The following McShane extension result can be found  in
\cite{mc}; see also \cite[Theorem 6.2]{he}.
Recall that a real-valued function $f$ on a metric space $(X,d)$
is \emph{$L$-Lipschitz} for some $L>0$ if,
for all $x$, $y$ in $X$, we have
$|f(x)-f(y)|\le L d(x,y)$.

\begin{lem}[McShane extension]\label{l2}
Let $A\subset X$ and $L\ge 1$. Suppose that $f:\,A\to \rr$ is an $L$-Lipschitz
function. Then the
function $F:\,X\to \rr$, given by $F(x):= \inf_{y\in A}\lf\{f(y)+Ld(x,y)\r\}$
for $x\in X$, is $L$-Lipschitz on $X$ with
$F=f$ on $A$.
\end{lem}

Now we are ready to prove Theorem \ref{t1.1}.

\begin{proof}[Proof of Theorem~\ref{t1.1}]
Due to Lemma \ref{l1}, it is not restrictive to assume that $g$ is Borel measurable and
\begin{equation*}\label{HG2}
|u(x)-u(y)|\le d(x,y)[g(x)+g(y)]
\end{equation*}
holds true for all  $x,\,y\in X$.

For each $k\in\nn:=\{1,2,\ldots\}$, denote $\{x\in X:\ g(x)\le 2^k\}$ by $E_k$.
Let $E:=\cup_{k\in\nn}E_k$ and $F:=X\setminus E$. Since $g\in L^1_\loc(X)$,
we know that $\mu(F) = 0$.  We will  complete the proof of Theorem~\ref{t1.1} via two steps.

{\bf Step 1.} For each non-constant rectifiable curve
$\gz: [a, b] \to X$ with $\gz(a),\, \gz(b)\in E$, we now show that
\begin{equation}\label{UG2}
|u(\gz(a))-u(\gz(b))|\le 4\int_\gz g(s)\,ds.
\end{equation}
Since $\gz(a),\, \gz(b)\in E$, there exists $k\in \nn$ such that $\gz(a),\, \gz(b)\in E_k$. Notice that, 
for all $x,\,y\in E_k$,
\[
|u(x)-u(y)|\le d(x,y)[g(x)+g(y)]\le 2^{k+1}d(x,y),
\]
that is, $u$ is $2^{k+1}$-Lipschitz on $E_k$. By Lemma~\ref{l2},
we extend $u$ to a $2^{k+1}$-Lipschitz function $u_k$ on $X$, defined by
setting, for all $x\in X$,
\begin{equation}\label{LipExt}
   u_k(x):= \inf_{y\in E_k}\lf\{u(y)+2^{k+1}d(x,y)\r\},
\end{equation}
where  $u_k = u$ on $E_k$. Let
\[
  g_k(x):= g(x)\chi_{E_k}(x)+2^{k+1}\chi_{X\setminus E_k}(x),\quad x\in X.
\]
It follows from the definition of $u_k$ and $g_k$ that,  for all $x,\, y\in E_k$,
\begin{equation}\label{uk}
|u_k(x)-u_k(y)|\le d(x,y)[g_k(x)+g_k(y)].
\end{equation}
If at least one of $x$ and $y$ belongs to $X\setminus E_k$,
then $g_k(x)+g_k(y)\ge 2^{k+1}$, and hence
\begin{eqnarray*}
|u_k(x)-u_k(y)|\le 2^{k+1}d(x,y)\le d(x,y)[g_k(x)+g_k(y)],
\end{eqnarray*}
which implies that \eqref{uk} holds for all $x,\,y\in X$.

We now  follow the proof of~\cite[Theorem~8.6]{h03} to show that $g_k$ acts like an upper gradient of $u_k$ on $\gamma$.
Notice that
$\gamma(a),\,\gamma(b)\in E_k$, and hence $u(\gamma(a))=u_k(\gamma(a))$ and
$u(\gamma(b))=u_k(\gamma(b))$.
Let $\gz:\, [a,b]\to X$ be parameterized by its arc-length. By Luzin's theorem there is a set $D\subset [a,b]$
 of full measure such that for each $t\in D$ there exists a sequence $h_n\to 0$ such that
 $g_k\circ\gz(t+h_n)\to g_k\circ\gz(t)$.
 As the function $u_k\circ\gz$ is Lipschitz from $[a,b]$ into $X$, for almost every $t\in D$, we have
\begin{eqnarray*}
\lf|\frac{\,d}{\,dt}u_k\circ\gz(t)\r|&&=\lf|\lim_{n\to\infty}\frac{u_k\circ\gz(t+h_n)-u_k\circ\gz(t)}{h_n}\r|\\
&&\le \limsup_{n\to\infty}\lf|\frac{d(\gz(t+h_n),\gz(t))}{h_n}\r|\lf[g_k\circ\gz(t+h_n)+ g_k\circ\gz(t)\r]\\
&&\le 2g_k\circ \gz(t).
\end{eqnarray*}
The fundamental theorem of calculus ensures that
\begin{equation}\label{UG_k}
|u_k(\gz(a))-u_k(\gz(b))|\le \int_a^b \lf|\frac{\,d}{\,dt}u_k\circ\gz(t)\r|\,dt \le 2\int_a^b g_k\circ \gz(t)\,dt.
\end{equation}

On the other hand, observe that, if $y\in E_k$, then $g_k(y)=g(y)$ and, if
$y\in X\setminus E_k$, then $g_k(y)=2^{k+1}<2g(y)$. Hence we obtain
\[
  |u(\gz(a))-u(\gz(b))|\le 4\int_\gz g(s)\,ds,
\]
completing the proof of Step~1.

{\bf Step 2.} It remains to show that,
if $\gz(a)$ or $\gz(b)$ lies in $F=X\setminus E$, inequality~\eqref{UG2}
also holds true.
We let $\wz u(x):= u(x)$ when $x\in E$ and, otherwise,
\[
  \wz u(x):=\limsup_{k\to \fz} u_k(x),
\]
where $u_k$ is as in~\eqref{LipExt}. Obviously, $\wz u=u$ almost everywhere.

If $\int_\gz g(s)\,ds$ is infinite, then~\eqref{UG2} holds trivially.
We next consider that $\int_\gz g(s)\,ds<\fz$. Notice that $g=\infty$ on $X\setminus E$,
and hence in this case
$\mathcal{H}^1(|\gz|\cap F)=0$,
where $\mathcal{H}^1$ denotes the $1$-dimensional Hausdorff measure.
Therefore, we may assume  that $\gz(a)\in X\setminus E$ and $\gz(b)\in E$,
by cutting $\gz$ into two sub-curves if necessary.

 For each $k \in \nn$, by~\eqref{UG_k}, we conclude that
\[
|u_k(\gz(a))-u_k(\gz(b))|\le 2\int_{\gz} g_k(s)\,ds\le 4\int_{\gz} g(s)\,ds.
\]
Since $\gz(b)\in E$, it follows that there exists $k_0\in\nn$ such that $u_k(\gz(b)) = u(\gz(b))$ for
each $k\ge k_0$. This further implies that, for each $k\ge k_0$,
\begin{equation}\label{x}
|u_k(\gz(a))-u(\gz(b))|=|u_k(\gz(a))-u_k(\gz(b))|\le 4\int_{\gz} g(s)\,ds
\end{equation}
and hence
\[
|u_k(\gz(a))|\le |u(\gz(b))|+4\int_{\gz} g(s)\,ds<\fz.
\]
Therefore, by \eqref{x}, we obtain
\begin{eqnarray*}
|\wz u(\gz(a))-\wz u(\gz(b))|
=\left|\limsup_{k\to\fz} u_k(\gz(a))-u(\gz(b))\right|
&&\le 4\int_{\gz} g(s)\,ds,
\end{eqnarray*}
which is the desired inequality.
The proof of Theorem~\ref{t1.1} is then completed by combining the above two steps.
\end{proof}

\section{Local Haj\l asz gradients}
\hskip\parindent
In this section, we show that, if $g$ is only a local Haj\l asz gradient (see the following definition) of $u$,
then the conclusion of Theorem~\ref{t1.1} still holds true.

\begin{defn}[Local Haj\l asz gradient]\label{lhg}
 Let $u$ be a measurable function on $X$.
A non-negative measurable function $g$ on $X$ is called a
\textit{local Haj\l asz gradient} of $u$ if, for each $z\in X$,
there exists an open set $U_z \ni z$ and  $E_z\subset U_z$ with $\mu(E_z)=0$ such that,
for all $x,\ y\in U_z\setminus E_z$,
\begin{equation}
|u(x)-u(y)|\leq d(x,y)[g(x)+g(y)].
\end{equation}
\end{defn}

Obviously, the Haj\l asz gradients of a measurable function are local Haj\l asz gradients of that function.
However, in general, local Haj\l asz gradients need not be global Haj\l asz gradients of a function.
For example, in the Euclidean setting $X=\mathbb{R}^2$, the function
$g=\chi_{B(0,9/4)\setminus B(0,3/4)}$ is a local Haj\l asz gradient of the function
$f$ given by $f(x)=1$ when $x\in B(0,1)$, $f(x)=0$ when $x\in\mathbb{R}^2\setminus B(0,2)$,
and $f(x)=1-\text{dist}(x,B(0,1))$ otherwise. However, $g$ cannot be a global Haj\l asz gradient
of $f$.

From the following Corollary~\ref{c3}, we shall see that local Haj\l asz gradients
can serve as a connection between Haj\l asz gradients and upper gradients.

For measurable functions $u,\,g$, we shall call
$g$ is a \textit{strong Haj\l asz gradient} of $u$ in the set $U\subset X$
if, for all $x,\ y\in U$, it holds true that
$
|u(x)-u(y)|\leq d(x,y)[g(x)+g(y)].
$

\begin{thm}\label{t3.1}
Suppose that $X$ is a separable space.
Let $u$, $g\in L^1_\loc(X)$. Suppose that $g$ is a local Haj\l asz gradient of $u$.
Then there exist $\wz u,\, \wz g\in L^1_\loc(X)$ such that $u=\wz u$ and $g=\wz g$ almost everywhere and
$4\,\wz g$ is an upper gradient of $\wz u$.
\end{thm}

\begin{proof}
As before, we can assume that $g$ is Borel measurable.

For each $x\in X$, we fix an open set $U_x\ni x$
such that $g$ is a Haj\l asz gradient of $u$ in $U_x$.
The collection $\{U_x\, :\, x\in X\}$ covers $X$;
hence by the separability of $X$ we find a countable subcover $\{U_j\}_j$ of $X$.
Furthermore, by modifying $g$ as in Lemma \ref{l1} on
a Borel set of measure zero if necessary, we can also assume that, for each $j$ and
for all $x,\, y\in U_j$,
\[
|u(x)-u(y)|\leq d(x,y)[g(x)+g(y)].
\]

Applying the proof of
Theorem~\ref{t1.1} to each $U_j$, we obtain a corrected (relative to $U_j$) function
$\widetilde{u}_j$ such that $\widetilde{u}_j=u$ almost everywhere in $U_j$ and $4g$ is an upper gradient
of $\widetilde{u}_j$ in $U_j$. Then clearly $\widetilde{u}_j=\widetilde{u}_m$ almost everywhere on the
set $U_j\cap U_m$; thus it only remains to show that, for $j\ne m$, the set
\[
Z_{j,m}:=\{x\in U_j\cap U_m:\ \widetilde{u}_j(x)\ne \widetilde{u}_m(x)\}
\]
has the property that $\mu(Z_{j,m})=0$ and,
whenever $\gamma$ is a rectifiable curve in $X$ intersecting $Z_{j,m}$,
then $\int_\gamma g(s)\, ds=\infty$. We can then set
$\wz u(x):=\widetilde{u}_j(x)$ for $x\in U_j\setminus \bigcup_m Z_{j,m}$, and
$\wz u(x):=0$ if $x\in Z_{j,m}$ for some $j\ne m$.

So let us assume that $\gamma$ is a non-constant compact rectifiable curve that intersects
$Z_{j,m}$. For each $x\in Z_{j,m}$, we know that either $\widetilde{u}_j(x)\ne u(x)$ or else
$\widetilde{u}_m(x)\ne u(x)$. Therefore, by the proof of Theorem~\ref{t1.1}, we know that
$g(x)=\infty$ for each $x\in Z_{j,m}$, and hence  $\mu(Z_{j,m})=0$.

Let $F:=\bigcup_{j,m} Z_{j,m}$. Then $g=\infty$ on $F$.
If $\int_\gamma g(s)\, ds<\infty$, we  must then have
$\mathcal{H}^1(|\gz|\cap F)=0,$
which implies that $|\gz|\cap E$ is a dense subset of $|\gz|$ in the $\mathcal{H}^1$ sense,
where $E:=X\setminus F$.
By breaking $\gamma$ into finitely many sub-curves if needed, it suffices to consider
only the case that $\gamma\subset U_j$. Fix $x\in |\gz|\cap Z_{j,m}$, and
take $\ez:=\frac 1{16}|\wz u_m(x)-\wz u_j(x)|$. Let $\gz:\,[a,c]\to U_j$ satisfying $\gz(a)=x$.
 By the absolute continuity of integrals and  $\mathcal{H}^1(|\gz|\cap F)=0$,
we conclude that there exists $d\in (a,c)$ such that $\gz(d)\in E$ and
\[
  \int_{\gz\vert_{[a,d]}}g(s)\, ds<\ez.
\]
Then, since
$\widetilde{u}_j(\gz(d))=\widetilde{u}_m(\gz(d))$ and $\gz(a)=x$, we know that
\begin{align*}
|\wz u_m(x)-\wz u_j(x)|&\le |\widetilde{u}_j(\gz(a))-\widetilde{u}_j(\gz(d))|+|\widetilde{u}_m(\gz(d))-\widetilde{u}_m(\gz(a))|\\
  &\le 8\int_{\gz\vert_{[a,d]}}g(s)\, ds <8\ez,
\end{align*}
which is impossible. Thus we
must conclude that $\int_\gz g(s)\, ds=\infty$.
This finishes the proof of Theorem~\ref{t3.1}.
\end{proof}

\begin{cor}\label{c3}
Suppose $u\in L^1_{\loc}(X)$ and let $D(u)$ be the collection of all
local Haj\l asz gradients of $u$ and $p\in(1,\infty)$. Then the closure of
$D(u)\cap L^p(X)$  in $L^p(X)$ is contained in the class of all weak upper
gradients of $u$ (up to a factor of $4$).
\end{cor}

A measure $\mu$ on $X$ is said to be \emph{locally doubling}
if, for each $R_0\in(0,\fz)$, there exists a positive constant
$C_d(R_0)$ such that, for each $r\in(0,R_0)$ and all $x\in X$,
\begin{equation*}
\mu(B(x,{2r}))\le C_d(R_0)\mu(B(x,r)).
\end{equation*}
We say that  $\mu$ is \emph{globally doubling} if the above inequality
holds with a  positive constant that is independent of $R_0.$

We also say that \emph{$X$ supports a local $p$-Poincar\'e inequality}
if, for each $R_0\in(0,\fz)$, there exists a positive constant
$C_P(R_0)$ such that, for each ball $B:=B(x,r)$ with $r\in (0,R_0)$ and  a function-upper gradient pair $(u,g)$ in $X$,
there exists $\lambda \in[1,\fz)$ so that
\begin{equation}\label{pi}
\frac{1}{\mu(B)}\, \int_B|u-u_B|\, d\mu
   \le C_P(R_0)\, r\, \left[\frac{1}{\mu(\lambda B)}\int_{\lambda B}g^p\, d\,\mu\right]^{1/p}.
\end{equation}
Here,
$
  u_B:=\frac{1}{\mu(B)}\, \int_B u\, d\mu
$
is the average of $u$ on the ball $B$, and $\lambda B$
the ball concentric with $B$ but with $\lambda$-times
the radius of $B$.
We then say that \emph{$(X,d,\mu)$ supports a $p$-Poincar\'e inequality},
if \eqref{pi} holds with a uniform positive constant $C_P$ for all $R_0\in(0,\fz)$.
For more information on Poincar\'e inequalities, we refer the reader to
\cite{hk98,hk00,ks08} and references therein.

Recall also that the \emph{restricted Hardy-Littlewood maximal function}
is defined for each
$f\in L^1_{\loc}(X)$ by
\[
  M_{R}f(x):=\sup_{0<t\le R} \frac{1}{\mu(B(x,t))}\, \int_{B(x,t)}|f(y)|\, d\mu(y),\quad x\in X.
\]

\begin{prop}\label{p3.1}
Let $X$ be complete, $\mu$ locally doubling,
and $X$ support a local $p$-Poincar\'e inequality for some $p\in(1,\infty)$.
Let $u\in L^1_{\loc}(X)$ and $D(u)$ be the collection of all local Haj\l asz gradients of $u$.
Then the closure of $D(u)\cap L^p(X)$ in $L^p(X)$ contains all the $p$-integrable
$p$-weak upper gradients of $u$ modulo a positive
constant multiple $C$.
\end{prop}

\begin{proof}
From the local doubling property of $\mu$, we know that, for each fixed $r\in (0,\infty)$,
 $M_r$ is a bounded operator on $L^t(X)$ whenever $t>1$.
Furthermore, by the results of~\cite{kz}, we know that there exists some $q\in (1,p)$
such that $X$ supports a local $q$-Poincar\'e inequality. Indeed, the proof of~\cite{kz}
carries through in the setting of local notions as long as $X$ is proper, with the
balls considered in the ``good $\lambda$-inequalities" of~\cite{kz} required in our  local
setting to be of radius smaller than $R_0$. The local
doubling property of $X$ together with the completeness of $X$ guarantees that closed balls,
and hence  closed and bounded subsets of $X$, are compact.

Since $p$-weak
upper gradients of $u$ can be approximated in $L^p(X)$ by upper gradients (see~\cite{km98}),
it suffices to show that upper gradients in $L^p(X)$ can be approximated by local Haj\l asz gradients.

Let $\rho\in L^p(X)$
be an upper gradient of $u$. Since $\mu$ is locally  doubling and $X$ supports a local
$q$-Poincar\'e inequality, we know that
there exists a positive constant $C$ such that, whenever $x,\,y\in X$ with $d(x,y)<R_0/2$
are Lebesgue points of $u$,
\[
  |u(x)-u(y)|\le C\, d(x,y)\,\left\{
    \left[M_{2\lambda d(x,y)}(\rho^{q})(x)\right]^{1/q}+\left[M_{2\lambda d(x,y)}(\rho^q)(y)\right]^{1/q}\right\};
\]
see, for example, \cite{hk00}.
It follows that, for $r\in(0,R_0/2)$,
$Cg_r:=[M_r(\rho^q)]^{1/q}$ is a local Haj\l asz gradient of $u$, with
$U_z=B(z,r/4)$ for each $z\in X$. Because $q<p$, we have $C\, g_r\in D(u)\cap L^p(X)$ based on
the assumption and, by the Lebesgue differentiation theorem, we know that $g_r\to \rho$ almost everywhere in $X$ as
$r\to 0$. It follows from the monotone convergence theorem that $g_r\to \rho$ in $L^p(X)$, that is,
$C^{-1}\rho$ is in the $L^p$-closure of $D(u)\cap L^p(X)$. This concludes the proof of
Proposition~\ref{p3.1}.
\end{proof}

Invoking the local Haj\l asz gradient,
one can introduce the corresponding \emph{Haj\l asz-Sobolev space $m^{1,p}(X)$},  for $p\in(1,\fz)$,
by replacing Haj\l asz gradients in the definition of $M^{1,p}(X)$ from the introduction with
 local Haj\l asz gradients.

The next corollary follows from Theorem~\ref{t3.1}, Proposition~\ref{p3.1},
and~\cite[Theorem 4.9]{s00}.
In what follows, $\hookrightarrow$ denotes continuous embedding.

\begin{cor}\label{c3.6}
Let $p\in [1,\fz)$.  Then $M^{1,p}(X)\hookrightarrow m^{1,p}(X)\hookrightarrow N^{1,p}(X)$.
If in addition $X$ is complete, $\mu$ is doubling, and $X$ supports a $p$-Poincar\'e inequality,
then $m^{1,p}(X)$, $M^{1,p}(X)$ and $N^{1,p}(X)$ coincide with equivalent norms.
\end{cor}

\begin{rem}
On an $n$-dimensional Riemannian manifold ${\mathscr M}$ with Ricci
curvature bounded from below by a negative constant, it is well known that
the measure and the Poincar\'e inequalities hold
only locally, we do not know whether $M^{1,p}({\mathscr M})$
and $N^{1,p}({\mathscr M})$ coincide or not, however,
the coincidence of $m^{1,p}({\mathscr M})$ and $N^{1,p}({\mathscr M})$
still holds true.
\end{rem}

\section{Some applications}
\hskip\parindent
In recent years,
there are many attempts to extend classical Sobolev type spaces to metric measure settings
based on some more general spaces besides Lebesgue spaces; see, for instance, \cite{Du,Tu,hhp,CJ,lyy}
for such developments. To best of our knowledge, the most general scale of Newtonian type spaces is due to Mal\'{y} \cite{maly1442,maly1448}, who
studied the Newtonian type spaces  associated with
a general quasi-Banach function lattice $\mathcal{B}$ on $X$.

Recall that a quasi-Banach space of real-valued measurable functions is a
vector space of real-valued functions equipped with a quasi-norm satisfying
Riesz-Fischer property (a notion analogous to completeness).
A \emph{quasi-Banach function lattice $\mathcal{B}$}
is a quasi-Banach space of real-valued measurable functions
on $X$ satisfying, if $f\in \mathcal{B}$
and $|g|\le |f|$ almost everywhere, then $g\in \mathcal{B}$ and $\|g\|_{\mathcal{B}}
\le \|f\|_{\mathcal{B}}$.

\begin{defn}[\cite{maly1442,maly1448}]
(i) Let $\wz{N}^{1}\mathcal{B}(X)$ be
the space of all measurable functions $u\in \mathcal{B}$ which have an
upper gradient $g\in \mathcal{B}$ and, for all $u\in \wz{N}^{1}\mathcal{B}(X)$, let
$$
\|u\|_{\wz{N}^{1}\mathcal{B}(X)}:= \|u\|_{\mathcal{B}}+\inf_g \|g\|_{\mathcal{B}},
$$
where the infimum is taken over all upper gradients $g$ of $f$.

The \textit{Newton-Sobolev space} ${N}^{1}\mathcal{B}(X)$ based on $\mathcal{B}$
is then defined as the quotient space ${N}^{1}\mathcal{B}(X):=\wz{N}^{1}\mathcal{B}(X)/\sim$,
equipped with
$\|\cdot\|_{{N}^{1}\mathcal{B}(X)}:=\|\cdot\|_{\wz{N}^{1}\mathcal{B}(X)}$,
where the relation $\sim$ is defined by $u\sim v$ if and only if
$\|u-v\|_{\wz{N}^{1}\mathcal{B}(X)}=0$.

(ii) The homogeneous version $\dot{N}^{1}\mathcal{B}(X)$ is defined via
replacing the condition  $u\in \mathcal{B}$ in the definition
of ${N}^{1}\mathcal{B}(X)$ by $u\in L^1_\loc(X)$ and the quasi-norm by
$$\|u\|_{\dot{N}^{1}\mathcal{B}(X)}:= \inf_g \|g\|_{\mathcal{B}},$$
where the infimum is taken over all upper gradients $g$ of $u$.
\end{defn}

Notice that ${N}^{1}L^p(X)=N^{1,p}(X)$. The
Haj\l asz-Sobolev spaces based on $\mathcal{B}$ can be defined as follows.

\begin{defn}\label{hbs}
(i)
The \textit{Haj\l asz-Sobolev space} ${M}^{1}\mathcal{B}(X)$ based on $\mathcal{B}$ is defined to be
the space of all measurable functions $u\in \mathcal{B}$ which have
a Haj\l asz gradient $g\in \mathcal{B}$ and
$$\|u\|_{M^{1}\mathcal{B}(X)}
:= \|u\|_{\mathcal{B}}+\inf_g \|g\|_{\mathcal{B}},$$
where the infimum is taken over all Haj\l asz gradients $g$ of $u$.

(ii) The homogeneous version $\dot{M}^{1}\mathcal{B}(X)$ is defined
via replacing the condition  $u\in \mathcal{B}$ in the definition of ${M}^{1}\mathcal{B}(X)$
by $u\in L^1_\loc(X)$ and the quasi-norm by
$$\|u\|_{\dot{M}^{1}\mathcal{B}(X)}:= \inf_g \|g\|_{\mathcal{B}},$$
where the infimum is taken over all Haj\l asz gradients $g$ of $u$.
\end{defn}

As an immediate consequence of Theorem~\ref{t1.1},
we obtain the following relations between
Haj\l asz-Sobolev spaces  and Newton-Sobolev spaces based on $\mathcal{B}$, the details being omitted.

\begin{thm}\label{t1}
Let $\mathcal{B}$ be a quasi-Banach function lattice on $X$ such
that $\mathcal{B}\subset L^1_\loc(X)$. Then ${M}^{1}\mathcal{B}(X)\hookrightarrow{N}^{1}\mathcal{B}(X)$ and $ \dot{M}^{1}\mathcal{B}(X)
\hookrightarrow\dot{N}^{1}\mathcal{B}(X).$
\end{thm}

In what follows, for all  $r\in(0,\fz)$, $u\in L^r_{\loc}(X)$ and $x\in X$, define
the \emph{non-centered un-restricted Hardy-Littlewood maximal operator $\mathcal{M}^{(r)}$ with
exponent $r$} by
\[
\mathcal{M}^{(r)}(u)(x):=\sup_{B\ni x}\lf\{\frac1{\mu(B)}\int_B |u(y)|^r\,d\mu(y)\r\}^{1/r},
\]
where the supremum is taken over all balls $B$ of $X$ containing $x$.

\begin{thm}\label{t2}
Let $p\in(1,\fz)$.
Assume that $X$ is complete, $\mathcal{B}$ is a quasi-Banach
function lattice on $X$ such that
$\mathcal{B}\subset L^1_\loc(X)$,
$X$ supports a $p$-Poincar\'e inequality for some $p\in(1,\fz)$, and that $\mu$
is doubling. If there exists $r_0\in(1,p)$
such that for all $r\in (r_0,p)$, the Hardy-Littlewood maximal function $\mathcal{M}^{(r)}$
is bounded on $\mathcal{B}$,
then
$${N}^{1}\mathcal{B}(X)={M}^{1}\mathcal{B}(X)\quad \mathrm{and}\quad
\dot{N}^{1}\mathcal{B}(X)
=\dot{M}^{1}\mathcal{B}(X)$$
with equivalent quasi-norms.
\end{thm}
\begin{proof}
By Theorem \ref{t1}, it suffices to prove the embeddings ${N}^{1}\mathcal{B}(X)\hookrightarrow{M}^{1}\mathcal{B}(X)$ and
$\dot{N}^{1}\mathcal{B}(X)
\hookrightarrow\dot{M}^{1}\mathcal{B}(X)$.
We only prove the first embedding due to similarity.

Let $u\in {N}^{1}\mathcal{B}(X)$. Then there exists an upper gradient
$g$ of $u$ such that
$\|u\|_{\mathcal{B}}+\|g\|_{\mathcal{B}}\le 2\|u\|_{{N}^{1}\mathcal{B}(X)}.$
Since $X$ supports a $p$-Poincar\'e inequality, by \cite[Theorem 1.0.1]{kz}, we know that
$X$ also supports a $r$-Poincar\'{e} inequality for some $r<p$ (which can be chosen large enough so that
$\mathcal{M}^{(r)}$ is also a bounded operator on $\mathcal{B}$) and hence,
by \cite[Theorem 3.2]{hk00},
there exist a positive constant $C$ and a set $E\subset X$ with
$\mu(E)=0$ such that, for all $x,\,y\in X\setminus E$,
\begin{equation*}
|u(x)-u(y)|\le Cd(x,y)\lf[\mathcal{M}^{(r)}(g)(x)+\mathcal{M}^{(r)}(g)(y)\r].
\end{equation*}
Therefore,  
$C\, \mathcal{M}^{(r)}(g)$ is a Haj\l asz gradient of $f$.
From the boundedness of $\mathcal{M}^{(r)}$ on $\mathcal{B}$, it follows that there exists a positive
constant $C$ such that, for all $u\in {N}^{1}\mathcal{B}(X)$,
$$\|u\|_{{M}^{1}\mathcal{B}(X)}\le C\lf\{\|u\|_{\mathcal{B}}+\|\mathcal{M}^{(r)}(g)\|_{\cb}\r\}
\le C\lf\{\|u\|_{\cb}+\|g\|_{\cb}\r\}\le C\|u\|_{{N}^{1}\mathcal{B}(X)},$$
and hence ${N}^{1}\mathcal{B}(X)\hookrightarrow{M}^{1}\mathcal{B}(X)$.
This finishes the proof of Theorem \ref{t2}.
\end{proof}

Finally, let us make some comments. Examples of $\mathcal{B}$ satisfying assumptions of
Theorems~\ref{t1} and~\ref{t2} include Lebesgue spaces (\cite{h96,s00,Du}), Lorentz spaces (\cite{CJ}),
Orlicz spaces (\cite{Tu}),  Orlicz-Musielak variable exponent space (\cite{hhp}),  and Morrey spaces ({\cite{lyy}}).

It is known that continuous functions are dense in
$M^{1,p}(X)$ and $N^{1,p}(X)$ with $p\in(1,\fz)$ if $X$ is a doubling space
supporting a $p$-Poincar\'e inequality (see \cite{s00}).
However, this density property might not be true for
the  spaces ${M}^{1}\mathcal{B}(X)$, ${\dot M}^{1}\mathcal{B}(X)$,
${N}^{1}\mathcal{B}(X)$ and ${\dot N}^{1}\mathcal{B}(X)$
when  $\mathcal{B}$ is a general quasi-Banach function lattice;
 see,
for example, \cite{lyy} for the case that $\mathcal{B}$ is a Morrey space.

The Newton-Morrey-Sobolev space $N^1M^q_p(X)$ and the
Haj\l asz-Morrey-Sobolev space $M^1M^q_p(X)$ with $1<p\le q<\fz$,
introduced and studied in \cite{lyy},  are our original motivation to compare Haj{\l}asz gradients and
upper gradients. Recall that the \emph{Morrey space} $M_p^q(X)$ with
$1<p\le q<\fz$  is defined as
the space of all measurable functions $f$ on $X$ such that
\begin{equation*}
\left\|f\right\|_{M_p^q(X)}:=\sup_{B\subset X}
[\mu(B)]^{1/q-1/p}\lf[\int_B|f(x)|^p\,d\mu(x)\r]^{1/p}<\fz,
\end{equation*}
where the supremum is taken over all balls $B$ in $X$.
It is shown by \cite[Remark 4.9]{lyy} that Lipschitz continuous functions are not dense in these spaces. Therefore,
to show the coincidence between $N^1M^q_p(X)$ and $M^1M^q_p(X)$,
we can not apply the  approach used in~\cite{s00}
for the coincidence between the
Newton-Sobolev and the Haj\l asz-Sobolev spaces.
However, Theorems~\ref{t1} and~\ref{t2} are applicable to these spaces,
since the maximal operator $\mathcal{M}^{(r)}$ is bounded on $M^q_p(X)$ when $r<p$
(see, for example, \cite{am}).

\medskip

{\small\noindent{\bf Acknowledgements}. We thank the referee for  valuable
suggestions that helped improve the exposition of this article.}

\bigskip

\noindent Renjin Jiang, Dachun Yang (Corresponding author) and Wen Yuan

\medskip

\noindent  School of Mathematical Sciences, Beijing Normal University,
Laboratory of Mathematics and Complex Systems, Ministry of
Education, Beijing 100875, People's Republic of China

\smallskip

\noindent {\it E-mails}: \texttt{rejiang@bnu.edu.cn} (R. Jiang)

\hspace{0.95cm}\texttt{dcyang@bnu.edu.cn} (D. Yang)

\hspace{0.85cm} \texttt{wenyuan@bnu.edu.cn} (W. Yuan)

\bigskip

\noindent Nageswari Shanmugalingam

\medskip

\noindent Department of Mathematical Sciences,
P.O.Box 210025, University of Cincinnati,
Cincinnati, OH 45221--0025, U.S.A.

\smallskip

\noindent{\it E-mail:} \texttt{shanmun@uc.edu}

\end{document}